\documentclass[a4paper,11pt]{article}
\usepackage{amsfonts}
\usepackage{amsthm}
\usepackage{amssymb}
\usepackage{latexsym}
\usepackage{graphics}
\usepackage{graphicx}
\usepackage{enumerate}
\usepackage{amsmath}
\usepackage{verbatim}
\usepackage{hyperref}
\usepackage{bibspacing}

\newtheorem{theorem}{Theorem}[section]
\newtheorem{lemma}[theorem]{Lemma}

\newtheorem{proposition}[theorem]{Proposition}

\theoremstyle{definition}

\newtheorem{remark}[theorem]{Remark}

\numberwithin{equation}{section}


\DeclareMathOperator{\diam}{diam}

\DeclareMathOperator*{\argmin}{argmin}

\newcommand{\norm}[1]{\left\Vert#1\right\Vert}
\newcommand{\snorm}[1]{\Vert#1\Vert}
\newcommand{\abs}[1]{\left\vert#1\right\vert}
\newcommand{\set}[1]{\left\{#1\right\}}
\newcommand{\brac}[1]{\left(#1\right)}
\newcommand{\scalar}[1]{\left \langle #1 \right \rangle}

\newcommand{\N}{\mathbb{N}}
\newcommand{\R}{\mathbb{R}}
\newcommand{\C}{\mathbb{C}}
\newcommand{\F}{\mathbb{F}}
\newcommand{\FF}{\mathcal{F}}
\newcommand{\E}{\mathbb{E}}

\newcommand{\eps}{\epsilon}

\newcommand{\Id}{{\rm Id}}

\newcommand{\G}{\mathcal{G}}
\newcommand{\PD}{{\rm PD}}
\newcommand{\SPD}{{\rm SPD}}
\newcommand{\GL}{{\rm GL}}
\newcommand{\M}{{\rm M}}

\renewcommand{\P}{\mathbb{P}}

\newlength{\defbaselineskip}
\newcommand{\setlinespacing}[1]           {\setlength{\baselineskip}{#1 \defbaselineskip}}

\renewcommand*{\thefootnote}{\fnsymbol{footnote}}

\begin{document}

\title{Regular Random Sections of Convex Bodies and the\\ Random Quotient-of-Subspace Theorem}

\medskip

\author{Emanuel Milman\textsuperscript{$*$,$\dagger$}
\and
Yuval Yifrach\textsuperscript{$*$,$\ddagger$}
}

\footnotetext[1]{Department of Mathematics, Technion-Israel Institute of Technology, Haifa 32000, Israel.}
\footnotetext[2]{Email: emilman@tx.technion.ac.il.}
\footnotetext[3]{Email: yuval@campus.technion.ac.il.}

\begingroup    \renewcommand{\thefootnote}{}    \footnotetext{2010 Mathematics Subject Classification: 46B07, 52A23, 46B20, 46B70.}
    \footnotetext{Keywords: Quotient-of-Subspace Theorem, Pisier's regular M-position, random Gelfand numbers.}

\date{} 
\maketitle

\begin{abstract}
It was shown by G.~Pisier that any finite-dimensional normed space admits an $\alpha$-regular $M$-position, guaranteeing not only regular entropy estimates but moreover regular estimates on the diameters of minimal sections of its unit-ball and its dual. We revisit Pisier's argument and show the existence of a \emph{different} position, which guarantees the same estimates for \emph{randomly sampled} sections \emph{with high-probability}. As an application, we obtain a \emph{random} version of V.~Milman's Quotient-of-Subspace Theorem, asserting that in the above position, \emph{typical} quotients of subspaces are isomorphic to Euclidean, with a distance estimate which matches the best-known deterministic one (and beating all prior estimates which hold with high-probability). Our main novel ingredient is a new position of convex bodies, whose existence we establish by using topological arguments and a fixed-point theorem.  
\end{abstract}

\maketitle

\section{Introduction}

Let $K$ denote an origin-symmetric convex compact set with non-empty interior (``body") in Euclidean space $(\R^n,\scalar{\cdot,\cdot})$. Note that origin-symmetric convex bodies are precisely the unit-balls of norms on $\R^n$. The polar body $K^{\circ}$ is thus defined as the unit-ball of the corresponding dual-norm, namely $K^{\circ} := \set{ y \in \R^n \; ; \; |\scalar{x,y}| \leq 1 \;\; \forall x \in K }$. Given another convex body $L$, recall that the covering number $N(K,L)$ is defined to be the minimal number of translates of $L$ whose union covers $K$, and that the entropy numbers $e_k(K,L)$ are defined as $\inf \set{ t > 0 \, : \, N(K,t L) \leq 2^{k-1}}$ for $k \in \N$. When $L = B_2^n$, the Euclidean unit-ball, we simply write $e_k(K) := e_k(K,B_2^n)$. 

In a seminal work \cite{Milman-M-pos-reverse-BM-CRAS} (cf. \cite{Milman-IsomorphicSurgery}), it was shown by V.~Milman that for every origin-symmetric convex body $K$ in $\R^n$, there exists a linear image (or position) $\tilde K$, so that:
\[
N(\tilde K,B_2^n) , N(\tilde K^{\circ}, B_2^n) , N(B_2^n , \tilde K) , N(B_2^n , \tilde K^\circ) \leq \exp(C n) ,
\]
for some universal numeric constant $C > 0$ (in fact, only the first two inequalities are required to hold, the last two inequalities follow from the first two by duality of entropy \cite{Konig-Milman, AMS-Duality-For-Ball, EMilman-Duality-of-Entropy}). Throughout this work, $c,C,C',C''$, etc.\ denote positive universal numeric constants, whose value may change from one instance to another. Any linear image $\tilde K$ satisfying the above covering estimates for a fixed constant $C>0$ is called an \emph{$M$-position of $K$}. The importance of the $M$-position stems (in particular) from its relation to Milman's reverse Brunn--Minkowski inequality \cite{Milman-IsomorphicSurgery} and the Bourgain--Milman reverse Blaschke--Santalo inequality \cite{Bourgain-Milman-vr-and-reverse-santalo} (see e.g. \cite[Chapter 8]{AGA-Book-I} or \cite[Chapter 7]{Pisier-Book}), and it has played a key role in numerous central results in the field (see \cite{AGA-Book-I,GreekBook,Pisier-Book} for subsequent developments). 

However, for many applications, one requires additional control over the covering numbers on all scales, with $B_2^n$ above replaced by $t B_2^n$. In another seminal work \cite{Pisier-Regular-M-Position}, the following refinement was obtained by G.~Pisier -- for any origin-symmetric convex body $K$ and for any $\alpha > 1/2$, there exists a position $\tilde K_\alpha$ so that:
\begin{equation} \label{eq:regular-M-pos}
N(\tilde K_\alpha, t B_2^n) , N(\tilde K_\alpha^{\circ}, t B_2^n) \leq \exp\brac{\frac{P_\alpha^{1/\alpha} n}{t^{1/\alpha}}} \;\;\; \forall t \geq 1 ,
\end{equation}
for some constant $P_\alpha$ depending solely on $\alpha$, satisfying:
\begin{equation} \label{eq:P-alpha}
P_\alpha \leq \frac{C}{\sqrt{\alpha-1/2}}  \text{ as $\alpha \rightarrow 1/2$} .
\end{equation}
For many applications it is vital to push $\alpha$ to the best possible value $\alpha = 1/2$, and so the behavior of $P_\alpha$ as $\alpha \rightarrow 1/2$ in (\ref{eq:P-alpha}) is of crucial importance. 
Any position satisfying (\ref{eq:regular-M-pos}) is called an \emph{$\alpha$-regular $M$-position} (with constant $P_\alpha$).

\medskip

In fact, Pisier established in \cite{Pisier-Regular-M-Position} a result which is \emph{stronger} than the covering-regularity estimates of (\ref{eq:regular-M-pos}),
pertaining to the regularity of diameter of sections of convex bodies and their polars.
To state it, define the Gelfand numbers $c_k(K)$, $k=1,\ldots,n$, which provide information on the minimal diameter of a section of $K$ by a $(k-1)$-codimensional subspace:
\[
c_k(K) := \inf \set{ \frac{1}{2} \diam(K \cap F) \; ; \; F \in G_{n,n-k+1} }  .
\]
Here $G_{n,m}$ denotes the Grassmannian of all $m$-dimensional linear subspaces of $\R^n$, equipped with its natural Haar probability measure $\sigma_{n,m}$. 
Note that $c_k(K^{\circ})$ provides dual information on the (reciprocal of the) maximal in-radius of an orthogonal projection of $K$ onto a $(k-1)$-codimensional subspace.
Pisier showed that for any origin-symmetric convex body $K$ and for any $\alpha > 1/2$, the position $\tilde K_\alpha$ satisfies:
\begin{equation} \label{eq:P-pos}
c_k(\tilde K_\alpha) , c_k(\tilde K^{\circ}_\alpha) \leq P_\alpha \brac{\frac{n}{k}}^{\alpha} \;\;\; \forall k = 1,\ldots,n  ,
\end{equation}
with $P_\alpha$ satisfying (\ref{eq:P-alpha}). The covering-regularity estimates (\ref{eq:regular-M-pos}) then follow from (\ref{eq:P-pos}) by an application of Carl's theorem (e.g. \cite[Theorem 5.2]{Pisier-Book}) to $L = \tilde K_\alpha,\tilde K^{\circ}_\alpha$, which implies that:
\[
\sup_{k \leq n} k^{\alpha} e_k(L) \leq C_\alpha \sup_{k \leq n} k^{\alpha} c_k(L) ,
\]
with $C_\alpha$ uniformly bounded for $\alpha$ on any compact subinterval of $(0,\infty)$. 

\subsection{$P$-typical-position}

We begin this work by noting
that while the covering-regularity estimates (\ref{eq:regular-M-pos}) are completely deterministic, 
the regularity estimates for sections and projections (\ref{eq:P-pos}) leave open the question of whether the same estimates hold, not only for a single good section or projection of $\tilde K_\alpha$, but also for a typical randomly lotterred section or projection (according to the natural Haar measure). An inspection of Pisier's argument (see Subsection \ref{subsec:comparison}) reveals that his proof only guarantees (for each $k$) the \emph{existence of a single} good section or projection of $\tilde K_\alpha$. 
In our first result in this work we revisit Pisier's proof, and observe that by applying some additional topological arguments, we can find a \emph{different position} of $K$ for which the estimates (\ref{eq:P-pos}) hold for random subspaces with high-probability:

\begin{theorem} \label{thm:main1}
For every origin-symmetric convex body $K \subset \R^n$ and every $\alpha > 1/2$, there exists a position $\bar K_\alpha$ so that:
\begin{equation} \label{eq:random-P-pos}
\sigma_{n,n-k+1}\set{ F \in G_{n,n-k+1} \; ; \; \diam(L \cap F) > \bar P_\alpha \brac{\frac{n}{k}}^{\alpha} } \leq \exp(- c k) ~,~ \forall L = \bar K_\alpha, \bar K^\circ_\alpha ~,~ \forall k=1,\ldots,n ,
\end{equation}
where $c > 0$ is a universal numeric constant and $P_\alpha = \bar P_\alpha$ satisfies (\ref{eq:P-alpha}). \end{theorem}

\noindent
To facilitate the ensuing discussion, we call $\tilde K_\alpha$ satisfying (\ref{eq:P-pos}) an \emph{$\alpha$-regular $P$-position} (with constant $P_\alpha$), and $\bar K_\alpha$ satisfying (\ref{eq:random-P-pos}) an \emph{$\alpha$-regular $P$-typical-position} (with constants $\bar P_\alpha, c$).

\smallskip

The main novelty in Theorem \ref{thm:main1} lies in the fact that $\bar P_\alpha$ satisfies (\ref{eq:P-alpha}) and not worse as $\alpha \rightarrow 1/2$. Indeed, as we shall presently explain, by employing various additional results from the literature, it is possible to deduce Theorem \ref{thm:main1} directly from (\ref{eq:regular-M-pos}) or from (\ref{eq:P-pos}) with $\bar K_\alpha = \tilde K_{\alpha}$ and a worse dependence, ensuring only: 
\begin{equation} \label{eq:bad-P-alpha}
  \bar P_\alpha \leq \frac{C}{(\alpha - 1/2)^{\frac{3}{2}}} \text{ as $\alpha \rightarrow 1/2$.}
\end{equation}
The main novelty of the proof consists of using a generalized fixed-point theorem on topological retracts to find a linear map $T_\alpha \in \GL_n$ so that (being slightly inaccurate) the complex interpolation of $\bar K_\alpha = T_\alpha(K)$ with $B_2^n$ is in the $\ell$-position 
-- see Proposition \ref{prop:new-pos}. Topological arguments have been previously used to find good positions using \emph{diagonal} maps $T$ (e.g. \cite[Lemma 3.1]{PaourisValettas-DeviationInNewPosition}), but we are not aware of any instance when a fully general map $T \in \GL_n$ has been specified using a topological argument. We also remark that when $K$ is the unit-ball of a normed space over the complex field $\C$, an analogous statement holds for typical $\C$-linear subspaces -- see the proof of Theorem \ref{thm:main1} in Section \ref{sec:main1}. 

\subsection{Relation to the $M$- and $P$-positions} \label{subsec:relation}

By the preceding remarks, we have the following sequence of implications:
\begin{multline*}
\text{$K$ is in $\alpha$-regular $P$-typical-position} \Rightarrow \\
\text{$K$ is in $\alpha$-regular $P$-position} \Rightarrow \\
\text{$K$ is in $\alpha$-regular $M$-position} .
\end{multline*}

It turns out that each of the above implications may be reversed, but at the price of increasing the constant $P_\alpha$. Following S.~Litvak and N.~Tomczak-Jaegermann \cite{LT-RandomGelfand}, it will be convenient to introduce the random Gelfand numbers $cr_k(K)$, defined for $k=1,\ldots,n$ as:
\[
cr_k(K) := \inf \set{ R > 0 \; ; \;  \sigma_{n,n-k+1}\{ F \in G_{n,n-k+1} \; ; \; \frac{1}{2} \diam(K \cap F) >  R \} \leq \exp(-c k) }  ,
\]
for an appropriate predetermined constant $c > 0$. By \cite[Lemma 3.3]{LT-RandomGelfand}, there exists $\rho > 0$ so that:
\[
\rho = cr_k(K \cap \rho B_2^n) = cr_k(K) .
\]
Recalling the argument from the proof of \cite[Theorem 3.2]{LT-RandomGelfand}, we employ the optimal form of Milman's low-$M^*$-estimate due to Pajor--Tomczak-Jaegermann \cite{PajorTomczakLowMStar} (cf. \cite[Theorem 7.3.1]{AGA-Book-I}), Dudley's entropy estimate \cite[Theorem 5.5]{Pisier-Book} (see e.g. the proof of \cite[Proposition 3.3]{EMilman-Duality-of-Entropy} for why it is enough to stop the summation at $j=n$), and Carl's theorem \cite[Theorem 5.2]{Pisier-Book}, to deduce that for any $k,m=1,\ldots,n$ and $s \in \{e,c\}$:
\[
\sqrt{k} \rho = \sqrt{k} \; cr_k(K \cap \rho B_2^n)  \leq C \ell^*(K \cap \rho B_2^n) \leq C' \sum_{j=1}^n \frac{s_j(K \cap \rho B_2^n)}{\sqrt{j}}  \leq C' \brac{\sum_{j=1}^m \frac{\rho}{\sqrt{j}} + \sum_{j=m+1}^n \frac{s_j(K)}{\sqrt{j}}}
\]
(see Section \ref{sec:prelim} for the definition of the Gaussian-mean dual-norm $\ell^*(L)$). Setting $m = \lfloor c k \rfloor$ for an appropriate $c > 0$, it follows that for $s \in \{e,c\}$:
\[
\sqrt{k} \; cr_k(K) = \sqrt{k} \rho \leq C'' \sum_{j = \lceil c k \rceil}^{n}  \frac{s_j(K)}{\sqrt{j}} .
\]
Consequently, if $K$ is either in $\alpha$-regular $P$-position or $\alpha$-regular $M$-position with constant $P_\alpha$, one readily checks that:
\[
cr_k(K) \leq \begin{cases} C 2^\alpha \frac{P_\alpha}{\alpha-1/2} \brac{\frac{n}{k}}^{\alpha} & \alpha > 1/2 \\ C P_{1/2} \log(1 + \frac{n}{k}) \sqrt{\frac{n}{k}}   & \alpha = 1/2 \end{cases} 
\]
(the case $\alpha=1/2$ above was obtained by V.~Milman in \cite[Theorem 3]{Milman-NoteOnLowMStar}).
In other words, $K$ is in $\alpha$-regular $P$-typical-position with constant $C 2^{\alpha} \frac{P_\alpha}{\alpha-1/2}$ whenever $\alpha > 1/2$. Recalling  (\ref{eq:P-alpha}), this establishes Theorem \ref{thm:main1} with $\bar K_\alpha = \tilde K_\alpha$ and the worse estimate (\ref{eq:bad-P-alpha}), so we will need a different approach to obtain (\ref{eq:P-alpha}).

For completeness, we also mention the high-dimensional paradigm that ``existence implies abundance", which in our context dictates that the existence of a single section of small diameter implies that a random section (of slightly higher codimension) also has a controlled diameter (see \cite{Vershynin-LocalVsGlobal, GMT-diameter, LPTJ-LocalVsGlobal}). However, even with the best known estimates, this would result in an additional $(n/k)^{1/2+\eps}$ factor in the diameter upper bound when transitioning from single to random section, which is much worse than the approach above. This is expected, as the former approach utilized the regularity of the entire sequence $\{e_k\}$ or $\{c_k\}$, not just a single value of $k$ as in the latter one. 

\subsection{Random Quotient-of-Subspace Theorem}

Our motivation for establishing Theorem \ref{thm:main1} is that it naturally leads to a random version of Milman's Quotient-of-Subspace (QS) Theorem \cite{Milman-QOS}. In \cite{Milman-QOS}, V.~Milman showed that for every normed space $X = (\R^n,\norm{\cdot})$ and any $k=1,\ldots,n$, there exists $E$, a quotient of a subspace of $X$, having dimension $m = n-k+1$, so that:
\[
d_{BM}(E,\ell_2^m) \leq \exp\brac{\exp\brac{C \frac{n}{k}}} .
\]
Here $d_{BM}(X,Y)$ denotes the Banach-Mazur distance between two normed spaces $X,Y$, defined as the infimum of $\snorm{T}_{op} \snorm{T^{-1}}_{op}$ over all linear isomorphisms $T : X \rightarrow Y$. Note that we do not distinguish between quotients of subspaces and subspaces of quotients as it is easy to check that these two classes coincide (see e.g. Lemma \ref{lem:perp}). Incorporating the optimal form of Milman's low-$M^*$-estimate due to Pajor--Tomczak-Jaegermann \cite{PajorTomczakLowMStar} into Milman's subsequent QS  argument from \cite{Milman-QOS-Better-low-MStar} (see \cite[Theorem 8.4]{Pisier-Book} or \cite[Theorem 7.4.1]{AGA-Book-I}), the following improved estimate holds:
\[
d_{BM}(E,\ell_2^m) \leq C \frac{n}{k} \log\brac{1 + \frac{n}{k}} .
\]
In geometric terms, the statement is equivalent to finding, for any $k=1,\ldots,n$, a position $\tilde K_k$ of an origin-symmetric convex body $K \subset (\R^n,\scalar{\cdot,\cdot})$ and subspaces $E \subset F \subset \R^n$ with $E$ of dimension $n-k+1$, so that:
\begin{equation} \label{eq:intro-QOS}
d_{G}((P_F \tilde K_k) \cap E , B_2^E) \leq C \frac{n}{k} \log\brac{1 + \frac{n}{k}} ,
\end{equation}
where $P_F$ denotes orthogonal projection onto $F$, $B_2^E$ denotes the Euclidean unit-ball in $E$, and $d_{G}(K,L)$ denotes the geometric distance between two origin-symmetric convex bodies $K,L$ in their linear hull, defined as $\inf \{ a b > 0 \; ; \; \frac{1}{b} L \subset K \subset a L\}$. To the best of our knowledge, (\ref{eq:intro-QOS}) is presently the best-known estimate in the QS Theorem. 

The original proofs in \cite{Milman-QOS,Milman-QOS-Better-low-MStar} involved a complicated iteration scheme, later simplified by Pisier (see \cite[Chapter 8]{Pisier-Book}), putting $K$ in a different position at each step of the iteration. As a result, the statement of the QS Theorem is deterministic, and only guarantees the \emph{existence} of the subspaces $E \subset F \subset \R^n$. Subsequent proofs of the QS Theorem, which guarantee that randomly selected $E \subset F \subset \R^n$ do the job with high-probability, yield estimates much worse than (\ref{eq:intro-QOS}). For instance (following the proof of \cite[Corollary 7.9]{Pisier-Book}), by putting $K$ in an $M$-position $\tilde K$, the orthogonal projection $P_F \tilde K$ onto a random subspace $F$ of dimension $n-\lceil k /2 \rceil +1$ is a finite-volume-ratio body with high-probability \cite[Theorem 8.6.1]{AGA-Book-I}, and thus by a theorem of Szarek and Tomczak-Jaegermann \cite[Theorem 5.5.3]{AGA-Book-I}, a random section $(P_F \tilde K) \cap E$ of dimension $n-k+1$ satisfies with high-probability:
\begin{equation} \label{eq:intro-random-QOS}
d_{G}((P_F \tilde K) \cap E , B_2^E) \leq e^{ C \brac{\frac{n}{k}}^2 } .
\end{equation}
While this may be satisfactory when $k$ is proportional to $n$, when $k \ll n$ the difference with (\ref{eq:intro-QOS}) is super exponential in $n/k$.
To the best of our knowledge, there were no prior results yielding (\ref{eq:intro-QOS}) with high-probability (even with a worse polylogarithmic term). Our second main result, which served as the impetus to Theorem \ref{thm:main1} and is a simple consequence thereof, closes this gap. For convenience, we replace $k$ by $2k$ in our formulation.

\begin{theorem} \label{thm:main2}
For every origin-symmetric convex body $K$ in $\R^n$ and for every integer $k \in [1,n/2]$, there exists a position $\bar K_k$ of $K$ so that:
\begin{align*}
\P\brac{d_{G}((P_F \bar K_k) \cap E , B_2^E)  \leq C \frac{n}{k} \log\brac{\frac{n}{k}}} \geq 1 - 2 \exp(-c k) , \\
\P\brac{d_{G}(P_E( \bar K_k \cap F) , B_2^E)  \leq C \frac{n}{k} \log\brac{\frac{n}{k}}} \geq  1 - 2 \exp(-c k) ,
\end{align*}
where $(F,E)$ is selected uniformly at random from the flag manifold $G^k = \{ (F,E) \; ; \; \R^n \supset F \supset E ~,~ F \in G_{n,n-k+1} ~,~ E \in G_{n,n-2k+2} \}$, 
namely $F$ is selected uniformly in $G_{n,n-k+1}$ and, given $F$, $E$ is selected uniformly in $G_{F,n-2k+2} = \{ E \in G_{n,n-2k+2} \; ;\; E \subset F\}$ (according to the corresponding Haar measures). \\
Moreover, this holds in an $\alpha$-regular $P$-typical-position $\bar K_k = \bar K_\alpha$ for $\alpha = \frac{1}{2} + \frac{1}{\log(n/k)}$, and we have:
\begin{equation} \label{eq:obviously}
\P\brac{ \diam(L \cap F) \leq C' \sqrt{\frac{n}{k} \log\brac{\frac{n}{k}}} \;} \geq 1 - \exp(- c k) \;\;\; \forall L = \bar K_k , \bar K_k^\circ . 
\end{equation}
\end{theorem}

We refer to Section \ref{sec:QOS} for more on the natural Haar measure on $G^k$. It is worthwhile to note the tradeoff between Theorem \ref{thm:main2} and the previous random versions of the QS Theorem: while the almost linear estimate $C \frac{n}{k} \log(\frac{n}{k})$ is attained with probability $1 - 2 \exp(-ck)$ in a position $\bar K_k$ which varies with $k$, the much worse super-exponential estimate in (\ref{eq:intro-random-QOS}) is achieved with probability $1 - \exp(-c n)$ in a position $\tilde K$ which works \emph{uniformly} for all $k$. A version of Theorem \ref{thm:main2} which holds uniformly in $k$, but yields the polynomial estimate $C (2 \alpha-1)^{-1} (\frac{n}{k})^{2 \alpha}$ with probability $1 - 2 \exp(-ck)$ in the $\alpha$-regular $P$-typical position $\bar K_\alpha$ for any given $\alpha > 1/2$, is stated in Theorem \ref{thm:uniform-QOS}; setting $\alpha = \frac{1}{2} + \frac{1}{\log(n/k)}$ yields the above optimized version.

\medskip

A final remark before concluding this Introduction. In view of F.~John's Theorem \cite[Chapter 2.1]{AGA-Book-I}, which implies that $d_{BM}(X,\ell_2^n) \leq \sqrt{n}$ for any $n$-dimensional normed space $X$ over $\R$, it is a very interesting question whether $\frac{n}{k}$ in any of the above results could be replaced by $\sqrt{\frac{n}{k}}$, as the latter is  the correct order of magnitude when $k$ is of the order of a constant. 

\bigskip

The rest of this work is organized as follows. In Section \ref{sec:prelim} we introduce some notation and provide the relevant background from Asymptotic Geometric Analysis. In Section \ref{sec:fixedpoint}, we prove the existence of a new position using topological methods. The proofs of Theorems \ref{thm:main1} and \ref{thm:main2} are provided in Sections \ref{sec:main1} and \ref{sec:QOS}, respectively.

\medskip
\noindent
\textbf{Acknowledgment.} We thank Bo'az Klartag for helpful discussions, and the referee for useful comments which have improved our presentation.

\section{Preliminaries} \label{sec:prelim}

In several places in this work we will need to work over the complex field $\C$, and so we start by carefully describing our setup over the field $\F$, where $\F$ is either $\R$ or $\C$. 
By considering the unit-ball of a given norm on $\F^n$, one obtains a well-known one-to-one correspondence between norms and origin-symmetric convex bodies in the case of $\R^n$, or circled convex bodies in the case of $\C^n$. 
Recall that a convex body $K$ in $\R^n$ is called origin-symmetric if $K = -K$, and a convex body in  $\C^n$ is called circled if $e^{i\theta} K = K$ for all $\theta \in [0,2\pi]$. We denote the norm whose unit-ball is the convex body $K$ by $\norm{\cdot}_K$ and the corresponding normed space by $X_K$. We equip $\F^n$ with the standard scalar-product $\scalar{\cdot,\cdot} = \scalar{\cdot,\cdot}_{\F}$ over $\F$, and define the polar body $K^{\circ}$ as the unit-ball of the dual-norm $\norm{\cdot}_K^*$ on $\F^n$, namely $K^{\circ} := \set{ y \in \F^n \; ; \; \abs{\scalar{x,y}_{\F}} \leq 1 \; \forall x \in K}$. We denote the Euclidean norm by $|x| = |x|_{\F} = \sqrt{\scalar{x,x}_{\F}}$. 

We will often identify $\C^n$ with $\R^{2n}$ by writing:
\[
\C^n = \{ x + i y \; ; \; x , y \in \R^n \} ,
\]
and modify our notation accordingly. To this end, it will be convenient to denote:
\[
N = \begin{cases} n & \F = \R \\ 2n & \F = \C \end{cases} ,
\]
so that $\F^n \simeq \R^N$. Note that if $u,v \in \C^n \simeq \R^{2n}$, then $\scalar{u,v}_{\R} = \text{Re} \scalar{u,v}_{\C}$; in particular, the Euclidean norm $|\cdot|_{\R}$ on $\R^{2n}$ agrees with the Euclidean norm $|\cdot|_{\C}$ on $\C^n$, and so there is no need to specify which one we employ. The Euclidean unit-sphere in $\F^n$ is denoted by $S^{N-1}$, and its associated Haar probability measure is denoted by $\sigma_N$.

Let $B_2^n = B_2^n(\F)$ denote the Euclidean unit-ball in $\F^n$. Given an $\R$-linear subspace $E$, we denote by $B_2^E$ the Euclidean unit-ball in $E$. 
If $K$ is a convex body in $E$, we denote by $r(K),R(K)$ the in-radius and out-radius of $K$, namely the best constants $r, R> 0$ so that:
\[
r B_2^E \subset K \subset R B_2^E  \;\;\;  (\text{equivalently, } \frac{1}{R} |x| \leq \norm{x}_K \leq \frac{1}{r} |x| \;\; \;\forall x \in E ) .
\]
The geometric distance $d_G(K,L)$ between two origin-symmetric convex bodies $K,L \subset E$ is defined as $\inf \{a b > 0\; ; \; \frac{1}{b} L \subset K \subset a L \}$.

The Grassmannian of all $m$-dimensional $\F$-linear subspaces of $\F^n$ ($0 \leq m \leq n$) is denoted by $G_{n,m}(\F)$; when $\F$ is clear from the context we will simply write $G_{n,m}$. We will denote the Grassmannian of all $m$-dimensional $\R$-linear subspaces of $\C^n$ ($0 \leq m \leq 2n$) by $G_{2n,m}(\R)$. $G_{n,m}(\F)$ is a homogeneous space of the orthogonal group $O(n)$ when $\F=\R$ and of the unitary group $U(n)$ when $\F= \C$, and so it is endowed with a natural (unique) Haar probability measure, which we denote by $\sigma_{n,m} = \sigma_{n,m}(\F)$. 

We denote by $\M_n(\F)$ the space of linear maps on $\F^n$. We denote the group of all $\F$-linear invertible maps on $\F^n$ by $\GL_n(\F)$. 
The group of $\R$-linear maps on $\C^n$ is denoted by $\GL_{2n}(\R)$. 
An invertible $\F$-linear image of a convex body $K$ in $\F^n$ is called a position of $K$. 
Note that $T(K)^{\circ} = T^{-*}(K^{\circ})$ for all $T \in \GL_n(\F)$. 

\subsection{$\ell$-norm and $\ell$-position} 

Let $G_N$ denote a standard Gaussian random vector on $\F^n$, namely $G_N = \sum_{i=1}^{N} g_i \xi_i$ where $\{\xi_i\}_{i=1}^N$ is an orthonormal basis of $(\F^n,\scalar{\cdot,\cdot}_\R)$ and $\{g_i\}_{i=1}^N$ are independent standard Gaussian random variables; its Gaussian distribution is denoted by $\gamma_N$. Given a normed space $X_K = (\F^n,\norm{\cdot}_K)$ and $p \in \{1,2\}$, we denote:
\[
\ell_p(K) := (\E \norm{G_N}_K^p)^{1/p} ~,~ \ell^*_p(K) = \ell_p(K^{\circ}) = (\E (\norm{G_N}^*_K)^p)^{1/p} ;
\]
when $p=1$, we simply write $\ell(K)$ and $\ell^*(K)$, respectively. Consider the minimization problem:
\begin{equation} \label{eq:l-pos-min}
\min \{\ell_2(T(K)) \; ; \; T \in \GL_n(\F)  \; ,\; \det T = 1 \} . 
\end{equation}
By e.g. (\ref{eq:al*}) below we can restrict the minimization to $\snorm{T^{-1}}_{op} \leq \sqrt{\pi/2} R(K) \ell_2(K)$ (where $\snorm{T^{-1}}_{op}$ refers to the operator norm of $T^{-1} : \ell_2^n(\F) \rightarrow \ell_2^n(\F)$). 
Compactness then implies that a minimizer $\tilde K = T(K)$ exists and is called an $\ell$-position of $K$; $K$ is said to be in $\ell$-position if the minimum above is attained on $T = \Id$. Denote by $\PD_n(\F)$ the convex cone in $\GL_n(\F)$ of positive-definite maps, and by $\SPD_n(\F)$ the subset of $\PD_n(\F)$ with determinant $1$. By invariance of $\gamma_N$ under the orthogonal group $O(N)$ and polar decomposition, the $\ell$-position may always be realized by a map $T$ in $\SPD_n(\F)$; moreover, this representative in $\SPD_n(\F)$ is necessarily unique (see the proof of Lemma \ref{lem:l-sym} below).

It was shown by Figiel and Tomczak-Jaegermann \cite{l-position} (cf. \cite[Theorem 6.4.5]{AGA-Book-I}) that if $K$ is in the $\ell$-position then:
\begin{equation} \label{eq:l-position}
\ell_2(K) \ell^*_2(K) \leq N \norm{Rad(X_K)}_{op} ,
\end{equation}
where $Rad(X_K) : L^2(\F^n,\gamma_{N}) \otimes X_K \rightarrow L^2(\F^n,\gamma_{N}) \otimes X_K$ is the (Gaussian) Rademacher projection and $\norm{Rad(X_K)}_{op}$ is its operator norm (also called the $K$-convexity constant of $X_K$). We refer to \cite[Chapter 2]{Pisier-Book} or \cite[Chapter 6]{AGA-Book-I} for precise definitions, as those will not be required here. At this point we only mention the following fundamental estimate of Pisier  \cite{Pisier-RadLogN} (cf. \cite[Theorem 2.5]{Pisier-Book},\cite[Appendix]{Bourgain-Milman-vr-and-reverse-santalo}): 
\begin{equation} \label{eq:Pisier-Rad}
\norm{Rad(X_K)}_{op} \leq C \log(1+n) .
\end{equation}
Combining (\ref{eq:l-position}) with (\ref{eq:Pisier-Rad}), we have (by Jensen's inequality) for any $K \subset \F^n$ in the $\ell$-position:
\begin{equation} \label{eq:ll*}
\ell(K) \ell^*(K) \leq C' n \log(1+n) .
\end{equation}

A point we do want to emphasize is that when $\F = \C$, the minimization in (\ref{eq:l-pos-min}) which guarantees (\ref{eq:l-position}) may indeed be taken only over $\GL_n(\C)$. 
Strictly speaking, since $\{\xi_i\}_{i=1}^{2n}$ is an orthonormal basis over $\R$, we do not know how to apply the Figiel--Tomczak-Jaegermann argument to the minimization problem (\ref{eq:l-pos-min}) over $\GL_n(\C)$, only over $\GL_{2n}(\R)$, 
and we can a-priori only conclude that the $\ell$-position which guarantees (\ref{eq:l-position}) is realized by a map $T \in \GL_{2n}(\R)$. However, as we've learned from Bo'az Klartag, the proper way to justify that things remain valid over $\C$ (when $K$ is a circled convex body) is to \emph{a-posteriori} deduce that $T \in \GL_n(\C)$. Since we will anyway require a more general version of this argument for later use, we formulate it as follows:

\begin{lemma} \label{lem:l-sym}
Let $K$ denote an (origin-symmetric) convex body in $\R^{N}$, and let $\G$ denote a subgroup of symmetries of $K$ in $O(N)$, namely $g K = K$ for all $g \in \G < O(N)$. 
Consider the minimization problem:
\begin{equation} \label{eq:l-pos-min2}
\min \{\ell_2(T(K)) \; ; \; T \in \GL_{N}(\R)  \; ,\; \det T = 1 \} . 
\end{equation}
Then there exists a unique minimizer $P$ of (\ref{eq:l-pos-min2}) in $\SPD_N(\R)$, and $P$ commutes with the elements of $\G$, namely $g P = P g$ for all $g \in \G$. 

In particular, if $K$ is a circled convex body in $\C^n \simeq \R^{N=2n}$, there exists $T \in \GL_n(\C)$ minimizing (\ref{eq:l-pos-min2}) and it is unique over $\SPD_n(\C)$. 
\end{lemma}
\begin{proof}
Let $T \in \GL_{N}(\R)$ be any minimizer of (\ref{eq:l-pos-min2}). By polar decomposition, we may write it as $U P$ with $U \in O(N)$ and $P \in \SPD_{N}(\R)$. By invariance of $\gamma_{N}$ under $O(N)$, it follows that $P$ is also a minimizer. Moreover, a minimizer of (\ref{eq:l-pos-min2}) in $\SPD_{N}(\R)$ is necessarily unique, as easily follows from the triangle inequality in $L^2$ and the fact that $\det^{1/{N}}$ is a concave function on $\PD_{N}(\R)$ which is strictly concave outside of rays. Since $g K = K$ for $g \in \G$ and since $\G < O(N)$, any $g^{-1} P g$ is also a minimizer of (\ref{eq:l-pos-min2}), and so by uniqueness over $\SPD_{N}(\R)$ it follows that $g^{-1} P g = P$, i.e. $g P = P g$.

It remains to note that when $K$ is circled in $\C^n \simeq \R^{2n}$, as $i K = K$, we may apply the above reasoning to $g = i$ and deduce that $i P = P i$.  But a $T \in \M_{2n}(\R)$ is in $\M_n(\C)$ iff $T$ commutes with $i$, and so we deduce that $P \in \M_{n}(\C)$ and hence in $\SPD_n(\C)$. 
\end{proof}

\medskip

Note that by a standard contraction principle \cite[(5.2.1)]{AGA-Book-I}, if $E$ is any $\R$-linear subspace of $\F^n$, then:
\[
\ell(K \cap E) \leq \ell(K) ~,~ \ell^*(P_E K) \leq \ell^*(K) ,
\]
where $P_E$ denotes orthogonal projection in $(\R^N,\scalar{\cdot,\cdot}_\R)$. By testing this on a one-dimensional $\F$-linear subspace $E$, it follows that:
\begin{equation} \label{eq:al*}
\frac{1}{r(K)} \leq \begin{cases} \sqrt{\pi/2} \; \ell(K) & \F = \R \\ \ell(K) & \F = \C \end{cases}  ~,~  
R(K) \leq \begin{cases} \sqrt{\pi/2} \; \ell^*(K) & \F = \R \\ \ell^*(K) & \F = \C \end{cases} .
\end{equation}

\subsection{Low-$M^*$-estimate}

Given an origin-symmetric convex body $K$ in $\F^n$ (which incidentally need not be circled), 
recall the definition of the random Gelfand numbers $cr^{\F}_k(K)$, defined for $k=1,\ldots,n$ as:
\[
cr^{\F}_k(K) := \inf \set{ R > 0 \; ; \;  \sigma_{n,n-k+1}\set{ F \in G_{n,n-k+1}(\F) \; ; \; R(K \cap F) >  R } \leq \exp(-c k) }  ,
\]
for an appropriate predetermined constant $c > 0$. These numbers provide information on the diameter of a typical section of $K$ by a $(k-1)$-codimensional \textbf{$\F$-linear} subspace (whose dimension is determined over $\F$). When $K \subset \C^n$, we may choose to also consider $cr^{\R}_k(K)$ corresponding  to $(k-1)$-codimensional $\R$-linear subspaces ($k=1,\ldots,2n$) of $\C^n \simeq \R^{2n}$:
\[
cr_k^{\R}(K)  = \inf \set{ R > 0 \; ; \;  \sigma_{2n,2n-k+1}\set{ F \in G_{2n,2n-k+1}(\R) \; ; \; R(K \cap F) >  R } \leq \exp(-c k) }.
\]

The optimal form of Milman's low-$M^*$-estimate due to Pajor--Tomczak-Jaegermann \cite{PajorTomczakLowMStar} (cf. \cite[Theorem 7.3.1]{AGA-Book-I}) asserts that:
\begin{equation} \label{eq:low-M*}
\sup_{k=1,\ldots,n} \sqrt{k} \; cr_k^{\F}(K) \leq C \ell^*(K) . 
\end{equation}
Note that this applies to the case $\F = \C$ as well with exactly the same proof (perhaps modifying the value of $C$). The reason is that $U(n)$ acts transitively on $S^{N-1}$ and thus the measure $\sigma_N$ is the the unique Haar probability measure under the actions of both $O(2n)$ and its subgroup $U(n)$; consequently the measures $\sigma_{n,n-k+1}(\C)$ on $G_{n,n-k+1}(\C)$ and $\sigma_N$ on $S^{N-1}$ are compatible, and all of the usual concentration estimates apply to $U(n)$ and $G_{n,n-k+1}(\C)$; of course, when constructing $\eps$-nets on subspaces and applying volumetric estimates, the pertinent dimensionality parameter is the one over $\R$ and not $\C$, but this factor of $2$ only results in a change of numeric constants. 

For completeness, we comment on the nomenclature ``low-$M^*$-estimate". It is easy to check using polar coordinates that $\ell^*(K) = (1+o(1)) \sqrt{n} M^*(K)$ where $M^*(K) = \int_{S^{N-1}} \norm{\theta}_K^* d\sigma_N(\theta)$. A lower bound on $M^*(K)$ in terms of $cr_k(K)$ was first obtained by Milman in \cite{Milman-Asterisque,Milman-QOS-Better-low-MStar}, but with worse dependence on $n/k$; the square-root dependence in (\ref{eq:low-M*}) established in \cite{PajorTomczakLowMStar} is best-possible. In fact, it was shown by Y.~Gordon \cite{Gordon-lowMStar} that the constant $C$ in (\ref{eq:low-M*}) may be taken to be arbitrarily close to $1$ \cite[Theorem 7.3.5]{AGA-Book-I}, and if one replaces $cr_k(K)$ by $c_k(K)$, one may actually use $C = 1+o(1)$ \cite[Theorem 7.7.1]{AGA-Book-I}.

\section{A new position via a fixed-point theorem} \label{sec:fixedpoint}

In the next section we will provide a proof of Theorem \ref{thm:main1} on the existence of an $\alpha$-regular $P$-typical-position. We do so by revisiting Pisier's argument from \cite[Theorem 7.13]{Pisier-Book}, but ultimately putting the body $K$ in a different position than Pisier. The main new challenge lies in justifying the existence of such a position, which is the goal of the present section.

Let $S = \{ z \in \C \; ; \; 0 < \text{Re}(z) < 1\}$ denote the unit complex strip, and let $\bar S$ denote its closure. We denote by $\FF_n$ the space of bounded continuous functions $f : \bar S \rightarrow \C^n$ which are holomorphic on $S$ (in each coordinate). Given two circled convex bodies $K_0,K_1$ in $\C^n$ and a parameter $\theta \in [0,1]$, the complex-interpolant body $[K_0 , K_1]_\theta$ is defined as:
\[
[K_0,K_1]_\theta := \{ f(\theta) \; ; \; f \in \FF_n ~,~ f(it)\in K_0 ~,~ f(1+it) \in K_1 \;\; \forall t \in \R \} .
\]
Clearly $[K_0,K_1]_0 = K_0$, $[K_0,K_1]_1 = K_1$, and $K_\theta := [K_0,K_1]_\theta$ is a circled convex body in $\C^n$ for all $\theta \in [0,1]$. It is easy to check that $K_\theta$ is precisely the unit-ball of the norm $\norm{\cdot}_\theta$ obtained by the complex-interpolation method of Calder\'on and Lions applied to the norms $\norm{\cdot}_0, \norm{\cdot}_1$ on $\C^n$ associated to $K_0,K_1$, respectively (cf. \cite{BL-InterpolationSpaces-Book,Pisier-Book}). The following are well-known properties of complex-interpolation (see e.g. \cite{Pisier-Book}):
\begin{align}
\label{eq:inter-inq}          & \norm{x}_\theta \leq \norm{x}_0^{1-\theta} \norm{x}_1^{\theta} \;\;\; \forall x \in \C^n , \\
\label{eq:inter-polar}      & [K_0,K_1]_\theta^{\circ} = [K_0^{\circ},K_1^{\circ}]_\theta , \\
\label{eq:inter-lin}           & T([K_0,K_1]_{\theta}) = [T(K_0),T(K_1)]_\theta \;\;\; \forall T \in \GL_n(\C) , \\
\label{eq:inter-ab}            & [a K_0,b K_1]_\theta = a^{1-\theta} b^{\theta} [K_0,K_1]_\theta \;\;\; \forall a,b > 0 . 
\end{align}
In addition, it will be crucial for us to note the following simple observation pertaining to the linear operator:
\[
\bar Z : \C^n \ni x+i y \mapsto x - i y \in \C^n  .
\]
\begin{lemma} \label{lem:barZ}
Assume that $K_0,K_1$ are invariant under $g \in U(n)$ or under $\bar Z \in O(2n)$. Then so is $[K_0,K_1]_\theta$ for all $\theta \in [0,1]$. 
\end{lemma}
\begin{proof}
If $g \in U(n)$ this is obvious from (\ref{eq:inter-lin}), or simply since the class of holomorphic functions $\FF_n$ is closed under linear transformations. To see this for $\bar Z$, note that $\FF_n$ is also closed under the mapping $f \mapsto \overline{f(\bar z)}$. 
\end{proof}

\medskip

Our main new observation in this section is the following:

\begin{proposition}  \label{prop:new-pos}
For any two circled convex bodies $K_0,K_1$ in $\C^n$ and $\theta \in [0,1)$, there exists a $T \in \PD_n(\C)$ so that $[T(K_0),K_1]_{\theta}$ is in $\ell$-position.

Moreover, if $\bar \G < O(2n)$ is the group generated by $U(n)$ and $\bar Z$, and $K_0,K_1$ are invariant under some subgroup $\G < \bar \G$,
then the above $T$ may be chosen to commute with all elements of $\G$, and hence $[T(K_0),K_1]_{\theta}$ is invariant under $\G$ as well.
\end{proposition}

Note that this is not at all obvious, since in general there is no explicit way to express $[T(K_0),K_1]_\theta$ by means of $T$ and $[K_0,K_1]_\theta$ (unless $T$ is a multiple of the identity, in which case (\ref{eq:inter-ab}) applies). We also remark that there is nothing special about the $\ell$-position in the above statement -- it equally holds for any position $\tilde K = T(K)$ which is invariant under $O(2n)$, uniquely defined over $\SPD_n(\C)$, varies continuously with $K$ and satisfies $r_n B_2^n \subset \tilde K \subset R_n B_2^n$ for some $r_n,R_n > 0$.

\medskip

For the proof, we will require a bit of background from topology. A topological space $\Omega$ is said to have the fixed-point property if any continuous $f : \Omega \rightarrow \Omega$ has a fixed-point $f(x) = x$. By Brouwer's theorem \cite{Munkres-TopologyBook2ndEd}, a compact ball in $\R^n$ has the fixed-point property, and hence so does any compact convex set in $\R^n$. However, we will avoid the question of whether the $\Omega$ we have in mind is homeomorphic to a compact ball, and so we will need to establish the existence of the fixed-point property by employing additional topological arguments from the theory of retracts, having its roots in the foundational work of Borsuk \cite{Borsuk-Retracts-Book}.

A subspace $\Omega$ of a topological space $Y$ is called a retract of $Y$ if there exists a continuous map $r : Y \rightarrow \Omega$ so that $r \circ \iota : \Omega \to \Omega$ is the identity map (where $\iota : \Omega \hookrightarrow Y$ is the inclusion map). Note that if $Y$ is Hausdorff, then $\Omega = \{y \in Y \; ; \; r(y) = y \}$ must be a closed subset of $Y$ \cite[I.(3.1)]{Borsuk-Retracts-Book}.

Observe that any retract $\Omega$ of a compact ball $B$ in $\R^n$ has the fixed-point property. Indeed, if $f : \Omega \rightarrow \Omega$ is continuous then $g := \iota \circ f \circ r : B \rightarrow B$ is continuous as well and has a fixed-point $y \in B$ by Brouwer's theorem; since $r \circ \iota$ is the identity map, it follows that $r(y) = r(g(y)) = f(r(y))$, and we confirm that $r(y) \in \Omega$ is a fixed-point for $f$. It remains to provide a good criterion for being a retract of a compact ball in $\R^n$. The following is known to experts:
\begin{lemma} \label{lem:retract} Any compact, contractible, and strictly locally contractible $\Omega \subset \R^n$ is a retract of a compact ball in $\R^n$ and hence has the fixed-point property. 
\end{lemma}
$\Omega$ is called strictly locally contractible if any open neighborhood $U$ of any point $x \in \Omega$ contains a contractible open neighborhood $V$ of $x$ (in fact, the weaker property of being locally contractible suffices for the statement above to hold, in which $V$ is only required to be contractible in $U$, not necessarily in itself). 
\begin{proof}
It was shown by Borsuk \cite[V.10.4]{Borsuk-Retracts-Book} (see \cite[V.10.7]{Borsuk-Retracts-Book} or \cite[V.7.1]{Hu-Retracts-Book} for the non-compact case) that $\Omega \subset \R^n$ is \emph{locally contractible} if and only if it is an Absolute Neighborhood Retract for the class of metrizable spaces ($\text{ANR}(\mathcal{M})$), meaning that whenever $\Omega$ is homeomorphic to a closed subset $\Omega'$ of a metrizable space $M$ then $\Omega'$ is a retract of an open neighborhood $N'$ of $\Omega'$ in $M$. In addition, Borsuk showed \cite[IV.9.1]{Borsuk-Retracts-Book} (cf. \cite[III.7.1-2]{Hu-Retracts-Book}) that a metrizable space $\Omega$ is a \emph{contractible} $\text{ANR}(\mathcal{M})$ if and only if $\Omega$ is an Absolute Retract for the class of metrizable spaces ($\text{AR}(\mathcal{M})$), meaning that whenever $\Omega$ is homeomorphic to a closed subset $\Omega'$ of a metrizable space $M$ then $\Omega'$ is a retract of $M$. 

We deduce by the above that any $\Omega \subset \R^n$ which is both locally contractible and contractible is an $\text{AR}(\mathcal{M})$. Consequently, if $\Omega$ is in addition closed then it is a retract of $\R^n$, and if $\Omega$ is in addition compact then it is a retract of a compact ball in $\R^n$. 
\end{proof}

We are now ready to prove Proposition \ref{prop:new-pos}:

\begin{proof}[Proof of Proposition \ref{prop:new-pos}]
Recall that $\SPD_n(\C)$ denotes the set of positive-definite maps on $\C^n$ having determinant $1$, which inherits the natural subspace topology from $\M_n(\C) \simeq \C^{n^2}$. Given $T \in \SPD_n(\C)$ and recalling Lemma \ref{lem:l-sym}, let $F(T)$ denote the unique element $S$ of $\SPD_n(\C)$ so that:
\[
 S([K_0,T^{-1}(K_1)]_\theta) \text{ is in $\ell$-position} . 
\] 
We will omit the external parentheses in the sequel. 
Since complex interpolation $[K_0,K_1]_\theta$ is clearly monotone in $K_0,K_1$ with respect to set-inclusion, $\SPD_n(\C) \ni T \mapsto K_T = [K_0,T^{-1}(K_1)]_\theta$ is continuous with respect to geometric distance on convex bodies. Consequently, the map
\[
 \SPD_n(\C) \times \SPD_n(\C) \ni (T,S) \mapsto \Lambda_2(T,S) := (\E \norm{G_N}^2_{S(K_T)})^{1/2} 
 \]
is jointly continuous in $(T,S)$,  and note that the $\ell$-position map $F(T)$ is defined as the \emph{unique} minimizer in $S$ of $\Lambda_2(T,\cdot)$. 
We will show below that $F$ maps an appropriate compact subset $\Omega \subset  \SPD_n(\C)$ into itself, and since $\Lambda_2$ is uniformly continuous on $\Omega \times \Omega$,  it follows that $\Omega \ni T \mapsto F(T) = \argmin \Lambda_2(T,\cdot)\in \Omega$ is necessarily continuous. 

For $T \in \SPD_n(\C)$, denote by $T_{\max},T_{\min} > 0$ the largest and smallest eigenvalues, respectively. 
Note that by (\ref{eq:al*}) and (\ref{eq:inter-inq}):
\[
 \ell^*(S [K_0,T^{-1}(K_1)]_\theta) \geq R(S [K_0,T^{-1}(K_1)]_\theta) \geq S_{\max} r([K_0,T^{-1}(K_1)]_\theta)  \geq S_{\max} r(K_0)^{1-\theta} (T_{\max}^{-1} r(K_1))^{\theta} ,
\]
and dually, using (\ref{eq:inter-polar}) and $(S^{-1})_{\max} = \frac{1}{S_{\min}}$:
\[
\ell(S [K_0,T^{-1}(K_1)]_\theta) \geq R(S^{-1} [K_0^{\circ},T(K_1^{\circ})]_\theta) \geq \frac{1}{S_{\min}} r([K_0^{\circ},T(K_1^{\circ})]_\theta) \geq \frac{1}{S_{\min}} r(K_0^{\circ})^{1-\theta} (T_{\min} r(K_1^{\circ}))^{\theta} .
\]
Since $S [K_0,T^{-1}(K_1)]_\theta$ is by definition in the $\ell$-position, we have by (\ref{eq:ll*}):
\[
\ell(S [K_0,T^{-1}(K_1)]_\theta) \ell^*(S [K_0,T^{-1}(K_1)]_\theta) \leq C n \log (1+n). 
\]
Multiplying the penultimate two inequalities and using that $R(K_i^{\circ}) = \frac{1}{r(K_i)}$, we conclude that:
\[
\frac{S_{\max}}{S_{\min}} \leq  C n \log(1+n) \brac{\frac{R(K_0)}{r(K_0)}}^{1-\theta} \brac{\frac{R(K_1)}{r(K_1)}}^{\theta}  \brac{\frac{T_{\max}}{T_{\min}}}^{\theta}   .
\]
Since $\theta \in [0,1)$, we see that for an appropriately large $M > 0$ we have:
\[
\frac{T_{\max}}{T_{\min}} \leq M \;\; \Rightarrow \;\; \frac{S_{\max}}{S_{\min}} \leq M . 
\]

Consequently, if we define $\Omega := \{ T \in \SPD_n(\C) \; ; \; T_{\max}/T_{\min} \leq M \}$, $\Omega$ is obviously compact and $F$ is a continuous map acting on $\Omega$.
Moreover, if $K_0,K_1$ are invariant under $\G$, we may further restrict the set $\Omega$ to only include those $T$'s which commute with $\G$ without affecting the above properties; indeed, if $T \in \Omega$ commutes with $\G$ then so does $T^{-1}$, hence $T^{-1}(K_1)$ is invariant under $\G$, and hence by Lemma \ref{lem:barZ} so is $[K_0,T^{-1}(K_1)]_\theta$, and hence $S$ which puts it in $\ell$-position commutes with $\G < O(2n)$ by Lemma \ref{lem:l-sym}, and so $F(T) \in \Omega$. 

 It is easy to see that $\Omega$ is a topological manifold-with-boundary and hence trivially strictly locally contractible. In fact, it is immediate to simultaneously verify contractiblity and strict local contractiblity around a given $T \in \Omega$, by using the contraction $\Omega \times [0,1]  \ni (S,t) \mapsto S_t := T^{t/2} S^{1-t} T^{t/2}$ and noting that $S_t \in \Omega$ for all $t \in [0,1]$ (as $\norm{A B}_{op} \leq \norm{A}_{op} \norm{B}_{op}$ and as functional calculus preserves commutation with normal operators). In summary, $\Omega \subset \M_n(\C) \simeq \C^{n^2}$ is contractible, strictly locally contractible and compact, and hence by Lemma \ref{lem:retract} has the fixed-point property. It follows that there exists $T \in \Omega$ so that $F(T) = T$, i.e. $T [K_0,T^{-1}(K_1)]_\theta = [T(K_0),K_1]_\theta$ is in $\ell$-position. This concludes the proof. 
\end{proof}

\section{Proof of Existence of $P$-Typical-Position} \label{sec:main1}

In this section we provide a proof of Theorem \ref{thm:main1} on the existence of an $\alpha$-regular $P$-typical-position.

One last ingredient we require is the following deep and crucial estimate of Pisier on the (Gaussian) Rademacher projection \cite[Theorem 7.11]{Pisier-Book} (cf. \cite{Pisier-ComplexInterpolation,Pisier-Type-Implies-K-Convex}) -- for any circled convex body $K$ in $\C^n$ and $\theta \in (0,1]$:
\begin{equation} \label{eq:inter-Rad}
K_\theta := [K,B_2^n]_{\theta} \;\; \Rightarrow \;\; \snorm{Rad_{X_{K_\theta}}}_{op} \leq \Phi(\theta) := \frac{1}{\tan(\pi \theta / 4)} .
\end{equation}
We mention in passing that from this estimate one may easily deduce the estimate (\ref{eq:Pisier-Rad}) -- see the remark following the proof of \cite[Theorem 7.11]{Pisier-Book}. 

\begin{proof}[Proof of Theorem \ref{thm:main1}] \hfill

\smallskip
\noindent
\textbf{Case of $\C^n$.}

We first treat the case that $K$ is a circled convex body in $\C^n$. Given $\alpha > 1/2$, set $\theta = 1 - \frac{1}{2\alpha} \in (0,1)$. By  Proposition \ref{prop:new-pos}, there exists 
$T = T_\alpha \in \PD_n(\C)$ so that $[T(K),B_2^n]_\theta$ is in $\ell$-position. By (\ref{eq:l-position}) and (\ref{eq:inter-Rad}), we obtain:
\[
\ell([T(K),B_2^n]_\theta) \ell^*([T(K),B_2^n]_\theta) \leq 2n \Phi(\theta) . 
\]
Consequently, by defining our desired position $\bar K_\alpha$ as $a T(K)$ for an appropriately chosen constant $a > 0$ and invoking (\ref{eq:inter-ab}) and (\ref{eq:inter-polar}), we ensure the individual estimates:
\begin{equation} \label{eq:scaling-T}
\ell^*([\bar K_\alpha,B_2^n]_\theta) \leq \sqrt{2 n \Phi(\theta)} ~,~ \ell^*([\bar K_\alpha^{\circ},B_2^n]_\theta) =  \ell([\bar K_\alpha,B_2^n]_\theta) \leq \sqrt{2 n\Phi(\theta)} .
\end{equation}
By the low-$M^*$-estimate (\ref{eq:low-M*}), we deduce:
\[
\sup_{k=1,\ldots,n} \sqrt{k} \; cr^{\C}_k([L,B_2^n]_\theta) ~,~ \sup_{k=1,\ldots,2n} \sqrt{k} \; cr_k^{\R}([L,B_2^n]_\theta) \leq C \sqrt{n} \sqrt{\Phi(\theta)} \;\;\; \forall L = \bar K_\alpha, \bar K_\alpha^{\circ} . 
\]
It remains to note that by (\ref{eq:inter-inq}), for any circled convex body $L$ and $\R$-linear subspace $E$:
\[
R([L,B_2^n]_\theta \cap E) \geq R(L \cap E)^{1-\theta} .
\]
Recalling that $\frac{1}{1 - \theta} = 2 \alpha$, it follows that:
\begin{equation} \label{eq:C-case}
\sup_{k=1,\ldots,n} k^{\alpha} \; cr^{\C}_k(L) ~,~ \sup_{k=1,\ldots,2n} k^{\alpha} \; cr^{\R}_k(L)  \leq C^{2 \alpha} n^\alpha \Phi(\theta)^{\alpha} \;\;\; \forall L = \bar K_\alpha, \bar K_\alpha^{\circ} .
\end{equation}
This confirms the assertion of Theorem \ref{thm:main1} for circled convex bodies $K$ in $\C^n$ with $\bar P_\alpha := C^{2 \alpha} \Phi( 1 - \frac{1}{2\alpha})^{\alpha}$. Recalling the definition of $\Phi$ in (\ref{eq:inter-Rad}), we immediately verify that $P_\alpha = \bar P_\alpha$ satisfies the asserted asymptotic behaviour (\ref{eq:P-alpha}) as $\alpha \rightarrow 1/2$. Note that we have confirmed that both typical $\C$-linear and typical $\R$-linear subspaces satisfy the above $\alpha$-regular estimates. 

\smallskip
\noindent
\textbf{Case of $\R^n$.}

It remains to treat the case when $K$ is an origin-symmetric convex body in $\R^n$. Following Pisier \cite[pp. 118-119]{Pisier-Book}, the idea is to reduce to the $\C^n$ case by complexifying $K$ into a circled convex body $K^{\C}$. Unfortunately, the reduction is not as direct as in \cite{Pisier-Book}, since we need to not only find a single projection of $K$ (in a suitable position) which has large in-radius, but rather confirm this for \emph{typical} orthogonal projections, and so we will need to use an additional symmetry of the map $T$ for which $[T(K^{\C}),B_2^n]_\theta$ is in $\ell$-position to ensure this. 

Let $X_K = (\R^n,\norm{\cdot}_K)$, and consider any complexification $X_K^{\C} = (\C^n , \norm{\cdot}_{K^{\C}})$ so that the inclusion $\iota : \R^n \hookrightarrow \C^n$ is an isometric embedding of $X_K$ into $X_K^{\C}$, 
and in addition $\norm{\cdot}_{K^{\C}}$ is invariant under $\bar Z$. Denoting:
\[
E_{\Re} = \iota(\R^n) = \{ x+i y \in \C^n \; ; \; y = 0 \} ,
\]
these two properties guarantee that $K^{\C} \cap E_{\Re} = K$ (obviously) but also $P_{E_\Re} K^{\C} = K$, since if $\norm{x+iy}_{K^{\C}} \leq 1$, then by $\bar Z$-invariance $\norm{x-iy}_{K^{\C}} \leq 1$, and so by the triangle inequality $\norm{x}_K = \norm{x}_{K^{\C}} \leq 1$ as well.

For instance, following Pisier, one may construct $X_K^{\C}$ by endowing $\C^* \otimes_{\R} X_K$ with the injective tensor product norm, i.e. $X_K^{\C}$ is the space of $\R$-linear operators $B$ from $\C$ to $X_K$ with the operator norm (and $z \cdot B = B \circ z$ for $z \in \C$).
 In more pedestrian terms, this means: 
 \[
\norm{x+ i y}_{X_K^\C} = \max_{\theta \in [0,2\pi]} \norm{\cos \theta \; x + \sin \theta \; y}_{K} = \max_{\theta \in [0, 2 \pi]} \snorm{ \text{Re} \; (e^{i\theta} \cdot (x+iy)) }_K ,
\]
where $\text{Re}(x + i y) = x$ is the natural orthogonal projection on $(\R^{2n},\scalar{\cdot,\cdot}_\R)$. 
Note that instead of using the $L^\infty([0,2\pi])$-norm above, we can equally use any $L^p([0,2\pi])$-norm (suitably normalized).

By Proposition \ref{prop:new-pos}, there exists $T \in \PD_n(\C)$ so that $[T(K^{\C}),B_2^n]_\theta$ is in $\ell$-position. Moreover, since $K^{\C}$ and $B_2^n$ are $\bar Z$-invariant, we may ensure  by Proposition \ref{prop:new-pos} that $T$ commutes with $\bar Z$ and hence $T(K^{\C})$ and $[T(K^{\C}),B_2^n]_\theta$ are $\bar Z$-invariant as well; these two latter properties will be crucial for us. 
Any $L \in \M_n(\C)$ may be written as $L = A + i B$ with $A,B \in \M_n(\R)$, and hence $L \in \M_n(\C)$ commutes with $\bar Z$ iff $L \in \M_n(\R)$ (i.e. $B = 0$). Consequently, our map $T$ from above is in fact in $\PD_n(\R)$, and $T(K^{\C}) \cap E_\Re = T( K^{\C} \cap E_\Re) = T(K)$. Since $T(K^{\C})$ is $\bar Z$-invariant, the argument from three paragraphs above implies that $P_{E_\Re} T(K^{\C}) = T(K^{\C}) \cap E_\Re = T(K)$. Hence $T(K^{\C})^{\circ} \cap E_\Re = (P_{E_\Re} T(K^{\C}))^{\circ} = T(K)^{\circ}$. All of the arguments involving the $\bar Z$-invariance were ultimately aimed at establishing that:
\begin{equation} \label{eq:aligned1}
T(K^{\C}) \cap E_\Re = T(K) \text{ and } T(K^{\C})^{\circ} \cap E_\Re = T(K)^{\circ} .
\end{equation}
We remark that in Pisier's original argument \cite[p. 119]{Pisier-Book}, $T(E_{\Re})$ is not guaranteed to coincide with $E_{\Re}$, and moreover $T(K^{\C})^{\circ} \cap T(E_\Re)$ will in general not coincide with $T(K)^{\circ}$. In other words, the map $T \circ \text{Re} \circ T^{-1}$ is not guaranteed to be an orthogonal projection onto $T(E_{\Re})$, and so the orthogonal projection of $T(K^{\C})$ onto $T(E_{\Re})$ will in general not coincide with $T(K)$. Consequently, we would not be able to identify between orthogonal projections of $T(K^{\C})$ onto subspaces $F \subset T(E_{\Re})$ and the corresponding orthogonal projections of $T(K)$. By using the $\bar Z$-invariance as above, we are able to overcome this obstacle. 

We are now ready to reduce to the complex case. By (\ref{eq:C-case}), setting $\bar K^{\C}_\alpha = a \, T(K^{\C})$ for an appropriate constant $a > 0$, we have for all $k=1,\ldots,2n$:
\begin{equation} \label{eq:real-prob1}
 \sigma_{2n,2n-k+1} \{ F \in G_{2n,2n-k+1}(\R) \; ; \; R(L^{\C} \cap F) > \bar P_\alpha \brac{\frac{n}{k}}^\alpha \} \leq \exp(-c k) \;\;\; \forall L = \bar K_\alpha, \bar K_\alpha^{\circ} .
\end{equation}
Setting $\bar K_\alpha = a \, T(K)$, (\ref{eq:aligned1}) translates to:
\[
L^{\C} \cap E_{\Re} = L \;\;\; \forall L = \bar K_\alpha , \bar K_\alpha^{\circ} ,
\]
and so for any $\R$-linear subspace $F$:
\[ R( L^{\C} \cap F )  \geq R( L \cap (F \cap E_{\Re}))   \;\;\; \forall L = \bar K_\alpha , \bar K_\alpha^{\circ} .
\] Given $L =\bar K_\alpha, \bar K_\alpha^{\circ}$, $k=1,\ldots,n$, it follows that for any $R > 0$:
\begin{align*}
        & \sigma_{2n,2n-k+1} \{ F \in G_{2n,2n-k+1}(\R) \; ; \; R(L^{\C} \cap F) > R \} \\
\geq \;\; & \sigma_{2n,2n-k+1} \{ F \in G_{2n,2n-k+1}(\R) \; ; \; R(L \cap (F \cap E_{\Re})) > R \} \\
 =     \;\; &  \sigma_{E_{\Re},n-k+1} \{ F \in G_{E_{\Re},n-k+1}(\R) \; ; \; R(L \cap F) >  R \} .
\end{align*}
The last equality holds since if $F$ is distributed according to $\sigma_{2n,2n-k+1} = \sigma_{2n,2n-k+1}(\R)$, then with probability $1$, $F \cap E_{\Re}$ is a $(k-1)$-codimensional $\R$-linear subspace of $E_{\Re}$  ($G_{E_{\Re},n-k+1}$ denotes the corresponding Grassmannian), and since its distribution is invariant under the action of the orthogonal group on $E_{\Re}$, it is necessarily distributed according to $\sigma_{E_{\Re},n-k+1}$, the (unique) Haar probability measure on $G_{E_{\Re},n-k+1}(\R)$. Thus, it immediately follows from (\ref{eq:real-prob1}) that for all $k=1,\ldots,n$:
\[
\sigma_{E_{\Re},n-k+1} \set{ F \in G_{E_{\Re},n-k+1}(\R) \; ; \; R(L \cap F) > \bar P_\alpha \brac{\frac{n}{k}}^\alpha } \leq \exp(-c k) \;\;\; \forall L = \bar K_\alpha, \bar K_\alpha^{\circ}  ,
\]
thereby concluding the proof. 
\end{proof}

\subsection{Comparison with Pisier's argument} \label{subsec:comparison}

It is worthwhile to compare the above argument with Pisier's original one from \cite{Pisier-Regular-M-Position} (cf. \cite[Theorem 7.13]{Pisier-Book}), 
in order to see why the latter only yields the existence of a \emph{single} good section (as opposed to with high-probability). Given two origin-symmetric convex bodies $K,L$ in a linear space $E$, let $R_L(K)$ denote the out-radius of $K$ w.r.t. the norm structure induced by $L$, namely $\max_{x \in E} \norm{x}_L/\norm{x}_K$. We sketch Pisier's argument in the complex case (using our geometric notation); the reduction from $\R^n$ to $\C^n$ is similar to the one described above.

 Let $K$ be a circled convex body in $\C^n$, and assume w.l.o.g. that it is in an $\alpha$-regular $P$-position (ensuring only \emph{existence} of sections) with best possible constant $P_\alpha$. Consider $[K,B_2^n]_{\theta}$ with $\theta = 1 - \frac{1}{2\alpha}$, and let $T \in \GL_n(\C)$ be such that $T([K,B_2^n]_{\theta})$ is in $\ell$-position (and scaled so that (\ref{eq:scaling-T}) holds). Then for any two subspaces $E_1,E_2 \in G_{n,n-k+1}$, denoting $E = E_1 \cap E_2$, we have by ``triangle inequality" and (\ref{eq:inter-inq}):
\begin{align*}
R_{T^{-1}(B_2^n) \cap E}(K \cap E) & \leq R_{T^{-1}(B_2^n) \cap E_1}([K,B_2^n]_{\theta} \cap E_1)  R_{[K,B_2^n]_{\theta} \cap E_2}(K \cap E_2) \\
& \leq R_{T^{-1}(B_2^n) \cap E_1}([K,B_2^n]_{\theta} \cap E_1)  R_{B_2^n \cap E_2}(K \cap E_2)^{\theta} . 
\end{align*}
By the low-$M^*$-estimate (\ref{eq:low-M*}), the first term on the right-hand side is bounded above by $C \sqrt{n/k} \sqrt{\Phi(\theta)}$ with high-probability on $E_1 \in G_{n,n-k+1}$ w.r.t. the Euclidean structure induced by $T$. As for the second term, there exists by definition $E_2 \in G_{n,n-k+1}$ so that $R_{B_2^n \cap E_2}(K \cap E_2) \leq P_\alpha (n/k)^{\alpha}$. Lastly, since $K$ was assumed to be in $\alpha$-regular $P$-position with best possible constant $P_\alpha$, it follows that there exists $k$ so that $R_{T^{-1}(B_2^n) \cap E}(K \cap E) \geq P_\alpha (n/(2k-1))^{\alpha}$ (perhaps after also applying the above argument to $K^{\circ}$). One can then conclude by solving the resulting equation for $P_\alpha$. 

We can attempt to modify the above argument by assuming that $K$ is in the $\alpha$-regular $P$-\emph{typical}-position (ensuring good sections with high-probability) with best possible constant $P_\alpha$. However, one quickly sees that the above bootstrap argument fails, since randomness over $E_2$ is taken w.r.t. the standard Euclidean structure, but randomness over $E_1$ is w.r.t. Euclidean structure induced by $T$, and so we cannot ensure that a good estimate holds for $E$ w.r.t. the standard Euclidean structure. Our fixed-point argument was precisely designed to address the latter incompatibility, and along the way slightly simplifies the overall argument since there is no longer any need for bootstrapping as above.

\section{Random Quotient-of-Subspace Theorem} \label{sec:QOS}

In this final section we prove a random version of the Quotient-of-Subspace Theorem. 

\medskip

We start with an elementary lemma:
\begin{lemma} \label{lem:perp}
Let $A$ be any subset of $\R^n$, and let $E_1,E_2$ denote two linear subspaces of $\R^n$. 
If $E_1 \supset E_2^{\perp}$ (equivalently $E_2 \supset E_1^{\perp}$) then $P_{E_1 \cap E_2} (A \cap E_1) = (P_{E_2} A)  \cap E_1$.
\end{lemma}
\begin{proof}
By definition:
\begin{align*}
P_{E_1 \cap E_2} (A \cap E_1) & = \{ x \in E_1 \cap E_2 \; ; \; \exists y \in (E_1 \cap E_2)^{\perp} \cap E_1 \;\; x + y \in A \} ,\\
(P_{E_2} A)  \cap E_1 & = \{ x \in E_1 \cap E_2 \; ; \; \exists y \in E_2^{\perp} \;\; x + y \in A \},
\end{align*}
and so the two sets are equal if $E_1 \cap (E_1 \cap E_2)^{\perp} = E_2^{\perp}$. But the latter holds iff $E_1 \supset E_2^{\perp}$. Indeed, the ``only if" direction is trivial, and if $E_1 \supset E_2^{\perp}$ then:
\[
E_1 \cap (E_1 \cap E_2)^{\perp} = E_1 \cap (E_1^{\perp} \oplus E_2^{\perp}) = E_1 \cap E_2^{\perp} = E_2^{\perp} . 
\]
\end{proof}

Recall the definition of the flag manifold $G^k = \{ (F,E) \; ; \; \R^n \supset F \supset E ~,~ F \in G_{n,n-k+1} ~,~ E \in G_{n,n-2k+2} \}$ ($k \leq n/2$). Note that $O(n)$ acts transitively on the compact manifold $G^k$, thereby equipping it with a unique Haar probability measure which is invariant under the latter action.
This means that if $(F,E)$ is uniformly distributed on $G^k$ according to its Haar probability measure, then $(U(F),U(E)) \sim (F,E)$ for any $U \in O(n)$,  where we use the notation $X \sim Y$ to signify that $X$ and $Y$ are identically distributed. Note that in particular, $U(F) \sim F$ for any $U \in O(n)$, and so by uniqueness of the Haar measure on $G_{n,n-k+1}$ it follows that $F$ is uniformly distributed on $G_{n,n-k+1}$. While we will not require this here, it is easy to show that conditioned on $F=F_0 \in G_{n,n-k+1}$, $E$ is uniformly distributed on the Grassmannian $\{ E \in G_{n,n-2k+2} \; ; \; E \subset F_0 \}$ (for $\sigma_{n,n-k+1}$-a.e. $F_0$). 

\medskip
We are now ready to state and prove the following version of Theorem \ref{thm:main2}.

\begin{theorem} \label{thm:uniform-QOS}
Given an origin-symmetric convex body $K$ in $\R^n$ and $\alpha > 1/2$, let $\bar K_\alpha$ denote an $\alpha$-regular $P$-typical-position of $K$ given by Theorem \ref{thm:main1}. Then for every integer $k \in [1,n/2]$:
\begin{align*}
\P(d_{G}((P_F \bar K_\alpha) \cap E , B_2^E)  & \leq C (2\alpha - 1)^{-1} \brac{n/k}^{2\alpha})  \geq 1 - 2 \exp(-c k) , \\
\P( d_{G}(P_E( \bar K_\alpha \cap F) , B_2^E) & \leq C (2\alpha - 1)^{-1} \brac{n/k}^{2\alpha})  \geq 1 - 2 \exp(-c k) ,
\end{align*}
where $(F,E)$ is selected at random according to the Haar measure on the flag manifold $G^k$. 
\end{theorem}

Clearly, setting $\alpha = \frac{1}{2} + \frac{1}{\log(n/k)}$ above yields Theorem \ref{thm:main2}. Note that (\ref{eq:obviously}) is then an immediate consequence of Theorem \ref{thm:main1}. 

\begin{proof}[Proof of Theorem \ref{thm:uniform-QOS}]
Given an integer $k \in [1,n/2]$, set $\bar R_{\alpha,k} := \bar P_\alpha (n/k)^{\alpha}$, where $\bar P_\alpha$ is the constant from Theorem \ref{thm:main1}. Recall the asymptotic behavior (\ref{eq:P-alpha}) of $P_\alpha = \bar P_\alpha$ as $\alpha \rightarrow 1/2$, which implies that:
\begin{equation} \label{eq:QOS1}
\bar R_{\alpha,k}^2 \leq C (2\alpha - 1)^{-1} \brac{n/k}^{2\alpha} ,
\end{equation}
for some universal constant $C > 0$. 

Denote:
\[
A_1 = \{ E \in G_{n,n-k+1} \; ; \; R(\bar K_\alpha \cap E) \leq \bar R_{\alpha,k} \} ~,~ A_2 = \{ E \in G_{n,n-k+1} \; ; \; R(\bar K_\alpha^{\circ} \cap E) \leq \bar R_{\alpha,k} \} .
\]
Note that if $E_1 \in A_1$, $E_2 \in A_2$ and $E_1 \supset E_2^{\perp}$ then:
\begin{equation} \label{eq:QOS2}
d_G(P_{E_1 \cap E_2}(\bar K_\alpha \cap E_1) , B_2^{E_1\cap E_2}) = d_G( P_{E_2}(\bar K_\alpha) \cap E_1 , B_2^{E_1\cap E_2})  \leq \bar R_\alpha^2 .
\end{equation}
Indeed (invoking Lemma \ref{lem:perp} in the last implication):
\begin{align*}
E_1 \in A_1 \;\; & \Rightarrow \;\; \bar K_\alpha \cap E_1 \subset \bar R_{\alpha,k} B_2^{E_1} \;\; \Rightarrow \;\; P_{E_1 \cap E_2}(\bar K_\alpha \cap E_1)   \subset \bar R_{\alpha,k} B_2^{E_1\cap E_2} , \\
E_2 \in A_2 \;\; & \Rightarrow \;\; P_{E_2} \bar K_\alpha \supset \frac{1}{\bar R_{\alpha,k}} B_2^{E_2} \;\; \Rightarrow \;\; P_{E_2}(\bar K_\alpha) \cap E_1 \supset \frac{1}{\bar R_{\alpha,k}} B_2^{E_1\cap E_2} , \\
E_1 \supset E_2^{\perp} \;\; & \Leftrightarrow \;\; E_2 \supset E_1^{\perp} \;\; \Rightarrow \;\; P_{E_1 \cap E_2}(\bar K_\alpha \cap E_1) = P_{E_2}(\bar K_\alpha) \cap E_1 .
\end{align*}

Now let $(F,E)$ be uniformly distributed on the flag manifold $G^k$, and set $E_1 = F$, $E_2 = F^{\perp} + E$. Then $E_1,E_2 \in G_{n,n-k+1}$, $E_1 \supset E_2^{\perp}$, and $E_1 \cap E_2 = E$. Recalling (\ref{eq:QOS2}) and (\ref{eq:QOS1}), it follows that:
\begin{equation} \label{eq:dg1}
\P( d_{G}(P_E( \bar K_\alpha \cap F) , B_2^E)  \leq C (2\alpha - 1)^{-1} \brac{n/k}^{2\alpha} ) \geq \P(E_1 \in A_1 \text{ and } E_2 \in A_2) . 
\end{equation}
Clearly, by reversing the roles of $E_1,E_2$ above, we equally have:
\begin{equation} \label{eq:dg2}
\P( d_{G}((P_F \bar K_\alpha) \cap E , B_2^E) \leq C (2\alpha - 1)^{-1} \brac{n/k}^{2\alpha} ) \geq \P(E_1 \in A_1 \text{ and } E_2 \in A_2) .
\end{equation}
It remains to bound $\P(E_1 \in A_1 \text{ and } E_2 \in A_2)$ from below. 

To this end, note that both $E_1$ and $E_2$ are (individually) uniformly distributed on $G_{n,n-k+1}$. Indeed, for $E_1$ this is obvious since $E_1 = F$ and we have already established the uniformity of $F$. To see this for $E_2$, note that for any $U \in O(n)$, $U(F^{\perp} + E) = U(F)^{\perp} + U(E) \sim F^{\perp} + E$, and the uniformity of $E_2$ follows by the uniqueness of the Haar probability measure on $G_{n,n-k+1}$. Consequently, applying Theorem \ref{thm:main1}, we obtain by the union bound:
\[
1 - \P(E_1 \in A_1 \text{ and } E_2 \in A_2) \leq \P(E_1 \notin A_1) + \P(E_2 \notin A_2) \leq 2 \exp(-c k) . 
\]
Recalling (\ref{eq:dg1}) and (\ref{eq:dg2}), this concludes the proof. 
\end{proof}

\begin{remark}
As explained in Subsection \ref{subsec:relation}, any $\alpha$-regular $M$-position or $\alpha$-regular $P$-position is an $\alpha$-regular $P$-typical position with a worse constant $\bar P_\alpha$. Consequently, by using (\ref{eq:bad-P-alpha}) and repeating verbatim the above proof, the statement of Theorem \ref{thm:main2} also holds with the worse estimate $C \frac{n}{k} \log^3(\frac{n}{k})$ when $\bar K_k$ is in Pisier's $\alpha$-regular $M$-position or $\alpha$-regular $P$-position with constant $P_\alpha$ satisfying (\ref{eq:P-alpha}) and $\alpha = \frac{1}{2} + \frac{1}{\log(n/k)}$.
\end{remark}

\setlinespacing{1.1}

\bibliographystyle{plain}
\bibliography{../../../ConvexBib}

\end{document}